\documentclass[11pt, reqno]{article}
\usepackage{a4wide}
\usepackage{tocloft}
\usepackage{amssymb}
\usepackage{amsmath}
\usepackage{amsthm}
\usepackage{tikz}
\usepackage{amscd}
\usepackage[curve]{xypic}
\usepackage{hyperref}
\usepackage[all]{xy}
\usepackage{color}
\theoremstyle{plain}
\usepackage[margin=1in]{geometry}

\newcommand{\id}{\operatorname{id}}

\newcommand{\sch}[1]{\operatorname{{\bf  #1}}}

\newcommand{\ind}{\operatorname{ind}}

\newcommand{\Char}{\operatorname{Char}}
\newcommand{\res}{\operatorname{res}}
\newcommand{\ho}{\operatorname{Hom}}

\newcommand{\img}{\operatorname{img}}

\newcommand{\Rep}{\operatorname{Rep}}
\newcommand{\Mod}{\operatorname{Mod}}

\newcommand{\End}{\operatorname{End}}
\newcommand{\Ext}{\operatorname{Ext}}

\newcommand{\bv}{\operatorname{{\bf v}}}

\newtheorem{theorem}{Theorem}[section]
\newtheorem{corollary}[theorem]{Corollary}
\newtheorem{lemma}[theorem]{Lemma}
\newtheorem{remark}[theorem]{Remark}
\newtheorem{proposition}[theorem]{Proposition}

\author{\large{ Santosh Nadimpalli }}
\date{\today}

\begin{document}
\title{On Extensions of supersingular representations of
  ${\rm SL}_2(\mathbb{Q}_p)$.}
\maketitle
\begin{abstract}
  In this note for $p>5$ we calculate the dimensions of
  $\Ext^1_{{\rm SL}_2(\mathbb{Q}_p)}(\tau, \sigma)$ for any two
  irreducible supersingular representations $\tau$ and $\sigma$ of
  ${\rm SL}_2(\mathbb{Q}_p)$. 
\end{abstract}
\section{Introduction}
In this note we calculate the space of extensions of supersingular
representations of ${\rm SL}_2(\mathbb{Q}_p)$ for $p>5$. The
dimensions of the space of extensions between irreducible
supersingular representations of ${\rm GL}_2(\mathbb{Q}_p)$ are
calculated by Pa\v{s}k\={u}nas in
\cite{paskunas_extensions}. Understanding extensions between
irreducible smooth representations play a crucial role in
Pa\v{s}k\={u}nas work on the image of Colmez Montreal functor in (see
\cite{paskunas_colmez_functor}). We hope that these results have
similar application to mod $p$ and $p$-adic local Langlands
correspondence for ${\rm SL}_2(\mathbb{Q}_p)$.

Let $G$ be the group ${\rm GL}_2(\mathbb{Q}_p)$, $K$ be the maximal
compact subgroup ${\rm GL}_2(\mathbb{Z}_p)$ and $Z$ be the center of
$G$. We denote by $I(1)$ the pro-$p$ Iwahori subgroup of $G$.  We
denote by $G_S$ the special linear group ${\rm
  SL}_2(\mathbb{Q}_p)$. For any subgroup $H$ of
${\rm GL}_2(\mathbb{Q}_p)$ we denote by $H_S$ the subgroup
$H\cap {\rm SL}_2(\mathbb{Q}_p)$. All representations in this note are
defined over vector spaces over $\bar{\mathbb{F}}_p$. Let $\sigma$ be
an irreducible smooth representation of $K$ and $\sigma$ extends
uniquely as a representation of $KZ$ such that $p\in Z$ acts
trivially. The Hecke algebra $\End_G(\ind_{KZ}^G\sigma)$ is isomorphic
to $\bar{\mathbb{F}}_p[T]$. For any constant $\lambda$ in
$\overline{\mathbb{F}}_p^{\times}$ let $\mu_\lambda$ be the unramified
character of $Z$ such that $\mu_{\lambda}(p)=\lambda$. Let
$\pi(\sigma, \mu_\lambda)$ be the representation
$$\dfrac{\ind_{KZ}^G\sigma}{T(\ind_{KZ}^G\sigma)}
\otimes(\mu_\lambda\circ\det).$$ The representations
$\pi(\sigma, \mu_\lambda)$ are irreducible (see
\cite{breuil_gl_2_mod-p_1}) and are called {\bf supersingular
  representations} in the terminology of Barthel--Livn\'e.

Let $\sigma_r$ be the representation ${\rm Sym}^r\bar{\mathbb{F}}_p$
of ${\rm GL}_2(\mathbb{F}_p)$. We consider $\sigma_r$ as a
representation of $K$ by inflation. The $K$-socle of
$\pi(\sigma_r, \mu_\lambda)$ is a direct sum of two irreducible smooth
representations $\sigma_r$ and $\sigma_{p-1-r}$. Let $\pi_{0,r}$ and
$\pi_{1, r}$ be the $G_S$ representations generated by
$\sigma_r^{I(1)}$ and $\sigma_{p-1-r}^{I(1)}$. The representations
$\pi_{0, r}$ and $\pi_{1,r}$ are irreducible supersingular
representations of $G_S$ and
$$\res_{G_S}\pi(\sigma_r, \mu_\lambda)\simeq \pi_{0,r}\oplus
\pi_{1,r}.$$ Any irreducible supersingular representation of $G_S$ is
isomorphic to $\pi_{i,r}$ for some $r$ such that $0\leq r\leq p-1$ and
$i\in \{0,1\}$. Moreover the only isomorphisms between $\pi_{i,r}$ are
$\pi_{0,r}\simeq \pi_{1,p-1-r}$ and $\pi_{1,r}\simeq \pi_{0,p-1-r}$
(see \cite{sl_2-mod_p-classification}). Our main theorem on extensions
of supersingular representations of $G_S$ is:
\begin{theorem}\label{main_theorem}
  Let $p\geq 5$ and $0\leq r\leq (p-1)/2$. For any irreducible
  supersingular representation $\tau$ of $G_S$ the space
  $\Ext_{G_S}^1(\tau, \pi_{i,r})$ is non-zero if and only if
  $\tau\simeq \pi_{j,r}$ for some $j\in \{0,1\}$. If
  $0\leq r< (p-1)/2$ then
  $\dim_{\bar{\mathbb{F}}_p}\Ext_{G_S}^1 (\pi_{i,r}, \pi_{j,r})=2$ for
  $i\neq j$ and
  $\dim_{\bar{\mathbb{F}}_p}\Ext_{G_S}^1 (\pi_{i,r},
  \pi_{i,r})=1$. For $r=(p-1)/2$ we have
  $\dim_{\bar{\mathbb{F}}_p}\Ext_{G_S}^1 (\pi_{0,r}, \pi_{0,r})=3$.
\end{theorem}
We briefly explain the method of proof. We essentially follow
\cite{paskunas_extensions}. The functor sending a smooth
representation to its $I(1)_S$-invariants induces an equivalence of
categories of smooth representations of $G_S$ generated by
$I(1)_S$-invariants and the module category of the pro $p$-Iwahori
Hecke algebra (see \cite[Theorem 5.2]{koziol_sl_2_equiv}). We use the
$\Ext$ spectral sequence thus obtained by this equivalence of
categories to calculate $\Ext^1_{G_S}$. Extensions of pro $p$-Iwahori
Hecke algebra modules are calculated from resolutions of Hecke modules
due to Schneider and Ollivier. We crucially use results from work of
Pa\v{s}k\={u}nas \cite{paskunas_extensions}. We first obtain lower
bounds on the dimensions of $\Ext^1_{G_S}$ spaces using the spectral
sequence and then obtain upper bounds using Pa\v{s}k\={u}nas results
on $\Ext^1_{K}(\sigma, \pi(\sigma, \mu_\lambda))$.

{\bf Acknowledgements} I thank Eknath Ghate for showing the
fundamental paper \cite{paskunas_extensions} and for his interest in
this work and discussions on the role of extensions in mod
$p$-Langlands. I want to thank Radhika Ganapathy for various
discussions on mod $p$ representations and for her mod $p$ seminar at
the Tata Institute.
 
\section{Pro-\texorpdfstring{$p$}{} Iwahori Hecke algebra}
Let $B$ be the Borel subgroup consisting of invertible upper
triangular matrices, $U$ be the unipotent radical of $B$ and $T$ be
the maximal torus consisting of diagonal matrices. We denote by
$\bar{U}$ the unipotent radical of $\bar{B}$ the Borel subgroup
consisting of invertible lower triangular matrices. We denote by $I$
the standard Iwahori-subgroup of $G$.  Let $I(1)$ be the pro-$p$
Iwahori subgroup of $G$ and $I(1)_S$ be the pro-$p$-Iwahori subgroup
of $G_S$. We note that $I(1)_S(Z\cap I(1))$ is equal to $I(1)$.  Let
$\mathcal{H}$ be the pro-$p$ Iwahori--Hecke algebra
$\End_{G}(\ind_{I(1)_S}^{G_S}\id)$. Let $\Rep_{G_S}$ and
$\Rep_{G_S}^{I(1)_S}$ be the category of smooth representations of
$G_S$ and its full subcategory consisting of those smooth
representations generated by $I(1)_S$-invariant vectors respectively.
We denote by $\Mod_{\mathcal{H}}$ the category of modules over the
ring $\mathcal{H}$. We have two functors
\begin{align*}
\mathcal{I}:\Rep_{G_S}^{I(1)_S}&\rightarrow \Mod_{\mathcal{H}}\\
&\mathcal{I}(\pi)=\pi^{I(1)_S}
\end{align*}
and
\begin{align*}
  \mathcal{T}:\Mod_{\mathcal{H}}&\rightarrow\Rep_{G_S}^{I(1)_S}\\ 
                    &\mathcal{T}(M)=M\otimes_{\mathcal{H}}\ind_{I(1)_S}^{G_S}\id.
\end{align*}
From \cite[Theorem 5.2]{koziol_sl_2_equiv} the functors $\mathcal{T}$
and $\mathcal{I}$ are quasi-inverse to each other. Let $\sigma$ and
$\tau$ be any two smooth representations of $G_S$ and $\sigma_1$ be
the $G_S$ subrepresentation of $\sigma$ generated by
$I(1)_S$-invariants of $\sigma$. We have
\begin{equation}
\ho_{G}(\tau, \sigma)=\ho_{G}(\tau,
\sigma_1)=\ho_{H}(\mathcal{I}(\tau), \mathcal{I}(\sigma_1))=
\ho_{H}(\mathcal{I}(\tau), \mathcal{I}(\sigma)).
\end{equation}
We get a Grothendieck spectral sequence with ${\rm E}_2^{ij}$ equal to
$\Ext^i(\mathcal{I}(\tau), \mathbb{R}^j\mathcal{I}(\sigma))$ such that 
\begin{equation}\label{basic_exact_sequence}
\Ext^i(\mathcal{I}(\tau),
\mathbb{R}^j\mathcal{I}(\sigma))\Rightarrow 
\Ext^{i+j}_{G}(\tau, \sigma).
\end{equation}
The $5$-term exact sequence associated to the above spectral sequence
gives the following exact sequence: 
\begin{align}\label{fiveterm}
0\rightarrow &\Ext^1_{\mathcal{H}}(\mathcal{I}(\tau),
  \mathcal{I}(\sigma))\xrightarrow{i} \Ext^1_{G}(\tau,
  \sigma)\xrightarrow{\delta}\\
\nonumber &\ho_{\mathcal{H}}(\mathcal{I}(\tau),
  \mathbb{R}^1\mathcal{I}(\sigma))\rightarrow
  \Ext^2_{\mathcal{H}}(\mathcal{I}(\tau),
  \mathcal{I}(\sigma))\rightarrow \Ext^2_{G}(\tau, \sigma)
\end{align}
for all $\tau$ such that $\tau=<G_S\tau^{I(1)}>$.  In particular we
apply these results when $\tau$ and $\sigma$ are irreducible
supersingular representations of $G_S$. We first recall the structure
of the ring $\mathcal{H}$, its modules $M(i, r)=\pi_{i, r}^{I(1)}$ for
$i$ in $\{0,1\}$ and $0\leq r\leq p-1$. The $\mathcal{H}$ module
$M(i,r)$ is a character and we first calculate the dimensions of the
spaces $\Ext^1_{\mathcal{H}}(M(i, r), M(j,s))$.

Let $T_S^0$ and $T_S^1$ be the maximal compact subgroup of $T_S$ and
its maximal pro-$p$-subgroup. We denote by $s_0$, $s_1$ and $\theta$
the matrices $\begin{pmatrix}0&1\\-1&0\end{pmatrix}$,
$\begin{pmatrix}0&-p^{-1}\\p&0\end{pmatrix}$ and
$\begin{pmatrix}p&0\\0&p^{-1}\end{pmatrix}$ respectively. Let $N(T_S)$
be the normaliser of the torus. The extended Weyl group
$W=\theta^{\mathbb{Z}}\coprod s_0\theta^{\mathbb{Z}}$ sits into an
exact sequence of the form
$$0\rightarrow \Omega:=\dfrac{T_S^0}{T_S^1}\rightarrow
\tilde{W}:=\dfrac{N(T_S)}{T_S^1}\rightarrow W=\dfrac{N(T_S)}{T_S^0}\rightarrow
0.$$ The length function $l$ on $W$, given by $l(\theta^i)=|2i|$ and
$l(s_0\theta^i)=|1-2i|$, extends to a function on $\tilde{W}$ such
that $l(\Omega)=0$. Let $T_w$ be the element $\Char_{I(1)wI(1)}$ for
all $w\in \tilde{W}$. We denote by $e_1$ the element
$\sum_{w\in \Omega}T_w$. The functions $T_w$ span $\mathcal{H}$ and
the relations in $\mathcal{H}$ are given by
\begin{align*}
T_wT_v&=T_{wv} \ \text{whenever} \ l(v)+l(w)=l(vw),\\
T_{s_i}^2&=-e_1T_{s_i}.
\end{align*}
The pro-$p$-Iwahori Hecke algebra is generated by $T_wT_{s_i}$ for $w$
in $\Omega$. For any character $\chi$ of $\Omega$ let $e_{\chi}$ be
the element $\sum_{w\in \Omega}\chi^{-1}(w)T_w$. Let $\gamma$ be a
$W_0$ orbit of the characters $\chi$ and $e_{\gamma}$ be the element
$\sum_{\chi\in \gamma}e_{\chi}$.  The elements
$\{e_{\gamma}; \gamma\in \hat{\Omega}/W_0\}$ are central
idempotents in the ring $\mathcal{H}$ and we have 
\begin{equation}\label{central_idempotents}
  \mathcal{H}=\bigoplus_{\hat{\Omega}/W_0}\mathcal{H}e_{\gamma}.
\end{equation}

For the group $G_S$, we know that $\mathcal{H}$ is the affine Hecke
algebra. The characters of affine Hecke algebra are described in a
simple manner we recall this for $G_S$.  Let $I$ be a subset of
$\{s_0, s_1\}$ and $W_I$ be the subgroup of $W$ generated by elements
of $I$ and $W_{\emptyset}$ is trivial group. The characters of
$\mathcal{H}$ are parametrised by pairs $(\lambda, I)$ where $\lambda$
is a character of $\Omega$ and $I\subset S_{\lambda}$. For such a pair
$(\lambda, I)$ the character $\chi_{\lambda, I}$ associated to it is
given by
\begin{align}\label{hecke_character}
  &\chi_{\lambda, I}(T_{wt})=0\ \text{for all}\  w\in W\backslash
    W_I\ \text{and}\ \text{for all} \ \ t\in \Omega,\\
  &\chi_{\lambda, I}(T_{wt})=\lambda(t)(-1)^{l(w)}\ \text{for all}\
    w\in W_I\ \text{and}\ \text{for all} \ \ t\in \Omega.
\end{align}
If $\lambda$ is nontrivial then we have
$\chi_{\lambda, \emptyset}(T_{t})=\lambda(t)$, for all $t\in \Omega$
and $\chi_{\lambda, \emptyset}(T_{wt})=0$ for all $w\neq \id$ and
$t\in \Omega$.

We denote by $\chi_{r, \emptyset}$ the character
$\chi_{x\mapsto x^r, \emptyset}$.  From the above description we get
that $M(0, r)=\chi_{r, \emptyset}$ and
$M(1,r)=\chi_{p-1-r, \emptyset}$ for $r\notin\{0, p-1\}$. If
$r\in \{0, p-1\}$ then \cite[Proposition 3.9]{canonical_torsion} says
that $\chi_{\id, \emptyset}$ and $\chi_{\id, S}$ are not supersingular
characters. This shows that $M(i, r)$ is either $\chi_{\id, I}$ or
$\chi_{\id, J}$, for $r\in \{0, p-1\}$, where $I=\{s_0\}$ and
$J=\{s_1\}$. Since the element $T_{s_0}$ belongs to pro-$p$
Iwahori--Hecke algebra of $G$ and using the presentation in
\cite[Corollary 6.4]{towards_mop_p_langlands} we obtain that
$M(1,0)=\chi_{\id, I}$ and $M(0, 0)=\chi_{\id, J}$. Similarly
$M(1,p-1)$ is given by the character $\chi_{\id, J}$ and $M(0, p-1 )$
is given by the character $\chi_{\id, I}$. Let $0\leq r,s\leq (p-1)/2$
then \eqref{central_idempotents} shows that
\begin{equation}\label{distinct_blocks}
\Ext^1_{\mathcal{H}}(M(i,r), M(j,s))=0
\end{equation} for $r\neq s$. 
\subsection{Resolutions of Hecke modules}
In order to calculate extensions between the characters $M(i,r)$, we
use resolutions constructed by Schneider and Ollivier for
$\mathcal{H}$. Let $\mathfrak{X}$ be the Bruhat--Tits tree of $G_S$
and let $A(T_S)$ be the standard apartment associated to $T_S$. We fix
an edge $E$ and vertices $v_0$ and $v_1$ contained in $E$ such that
the $G_S$-stabiliser of $v_0$ is $K_S$.  For any facet $F$ of
$\mathfrak{X}$ we denote by $\sch{G}_F$ the $\mathbb{Z}_p$-group
scheme with generic fibre $\sch{SL}_2$ and $\sch{G}_F(\mathbb{Z}_p)$
is the $G$-stabiliser of $F$. We denote by $I_F$ the subgroup of
$\sch{G}_F(\mathbb{Z}_p)$ whose elements under mod-~$\mathfrak{p}$
reduction of $\sch{G}_F(\mathbb{Z}_p)$ belong to the
$\mathbb{F}_p$-points of the unipotent radical of
$\sch{G}_F\times \mathbb{F}_p$. We denote by $\mathcal{H}_F$ the
finite subalgebra of $\mathcal{H}$ defined as
$$\mathcal{H}_F:=
\End_{\sch{G}_F(\mathbb{Z}_p)}(\ind_{I_F}^{\sch{G}_F(\mathbb{Z}_p)}(\id)).$$
In particular $\mathcal{H}_{E}$ is a semi-simple algebra.

For any $\mathcal{H}$-module $\mathfrak{m}$ the construction of
Schneider and Ollivier \cite[Theorem 3.12,
(6.4)]{pro-p-hecke-gorenstein} gives us a
$(\mathcal{H}, \mathcal{H})$-exact resolution
\begin{equation}\label{Hecke_resolution}
  0\rightarrow \mathcal{H}\otimes_{\mathcal{H}_E}\mathfrak{m}
  \xrightarrow{\delta_1}
  (\mathcal{H}\otimes_{\mathcal{H}_{v_0}}\mathfrak{m})\oplus
  (\mathcal{H}\otimes_{\mathcal{H}_{v_1}}\mathfrak{m})
  \xrightarrow{\delta_0}\mathfrak{m}\rightarrow 0.
\end{equation} 
Using the resolution \eqref{Hecke_resolution} and the observation that
$\mathcal{H}_E$ is semi-simple for $p\neq 2$ we get that
\begin{equation}\label{ext_two_points}
  0\rightarrow 
  \ho_{\mathcal{H}}(\mathfrak{m}, \mathfrak{n})\rightarrow
  \bigoplus_{v_0,v_1}\ho_{\mathcal{H}_{v_i}}(\mathfrak{m},
  \mathfrak{n})\rightarrow\\
  \ho_{\mathcal{H}_E}(\mathfrak{m},
  \mathfrak{n})\xrightarrow{\delta}
  \Ext^1_{\mathcal{H}}(\mathfrak{m},
  \mathfrak{n})\rightarrow \bigoplus_{v_0,
    v_1}\Ext^1_{\mathcal{H}_{v_i}}(\mathfrak{m},
  \mathfrak{n})\rightarrow 0
\end{equation}
Note that we have an isomorphism of algebras
$$\mathcal{H}_{v_0}\simeq
\mathcal{H}_{v_1}\simeq \End_{{\rm
    SL}_2(\mathbb{F}_p)}(\ind_{N(\mathbb{F}_p)}^{{\rm
    SL}_2(\mathbb{F}_p)}\id).$$ The above isomorphism is not a
canonical isomorphism. Let $K_0$ and $K_1$ be the compact open
subgroups $K\cap G_S$ and $K^{\Pi}\cap G_S$ respectively. 
\subsection{Extensions of supersingular modules over
  pro-\texorpdfstring{$p$}{} Iwahori--Hecke algebra.}
The Hecke algebra $\mathcal{H}_{v_i}$ is isomorphic to
$\End_{K_i}(\ind_{I(1)}^{K_i}\id)$. The Hecke algebra
$\mathcal{H}_{v_i}$ is generated by $T_t$ and $T_{s_i}$ for
$t\in \Omega$.  The relations among them are given by
\begin{align*}
&T_{t_1}T_{t_2}=T_{t_1t_2},\\
             & T_tT_{s_i}=T_{ts_i}=T_{s_it^{-1}}=T_{s_i}T_{t^{-1}},\\
             & T_{s_i}^2=-e_1T_{s_i}
\end{align*}
where $e_1=\sum_{t\in \Omega}T_t$.
\begin{lemma}\label{regular_hecke_ext}
  Let $0<r\leq (p-1)/2$ the space
  $\Ext^1_{\mathcal{H}}(M(i,r), M(j,s))$ is non-zero if and only if
  $i\neq j$ and has dimension $2$ when $i=j$. If $r=(p-1)/2$ then the
  space $\Ext^1_{\mathcal{H}}(M(i,r), M(i,r))$ has dimension $2$.
\end{lemma}
\begin{proof}
  Since $r\neq 0$ the characters $M(0,r)$ and $M(1,r)$ are isomorphic
  to $\chi_{r, \emptyset}$ and $\chi_{p-1-r, \emptyset}$ respectively
  (see \eqref{hecke_character}).  Let $E_c$ be a $2$-dimensional
  $\bar{\mathbb{F}}_p$ module
  $\bar{\mathbb{F}}_pe_1\oplus \bar{\mathbb{F}}_pe_2$ and
  $\bar{\mathbb{F}}_p[\Omega]$ acts on $E$ by $T_te_0=t^re_0$ and
  $T_te_1=t^{p-1-r}e_1$.  We set $T_{s_i}e_0=0$ and $T_{s_i}e_1=ce_0$
  for some $c\neq 0$. This makes $E$ a $\mathcal{H}_{v_i}$ module and
  is a non-trivial extension
  $$0\rightarrow\chi_{r, \emptyset}\rightarrow E\rightarrow \chi_{p-1-r,
    \emptyset}\rightarrow 0.$$ Let $E$ be a
  $\mathcal{H}_{v_i}$-extension of $W:=\chi_{s, \emptyset}$ by
  $V:=\chi_{r, \emptyset}$ i.e, we have an exact sequence
  $$0\rightarrow V\rightarrow E\xrightarrow{f} W \rightarrow0.$$
  There exists a $\bar{\mathbb{F}}_p[\Omega]$-equivariant section
  $s:W\rightarrow E$ of $f$. Let $V'$ be the image of this section.
  Now $E=V\oplus V'$. The action of $T_{s_i}$ is trivial on $V$ and
  observe that $f(T_{s_i}(V'))=T_{s_i}(f(V'))=0$. {\bf If $E$ is
    nontrivial then $T_{s_i}(V')=V$}. This implies that $r+s=p-1$ and
  hence $E$ is isomorphic to $E_c$ for some $c\neq 0$. This shows that
  the space of $\mathcal{H}_{v_i}$ extensions of $W$ by $V$ is one
  dimensional if $r+s=p-1$ and zero otherwise.  Now consider the exact
  sequence \eqref{ext_two_points} when $\mathfrak{m}$ is $M(i,r)$ and
  $\mathfrak{n}$ is $M(j,r)$.  For $i=j$ the map $\delta$ in zero
  \eqref{ext_two_points} hence the space
  $\Ext^1_{\mathcal{H}}(M(i,r), M(i,r))$ is trivial. When $i\neq j$
  the $\ho$ spaces in \eqref{ext_two_points} are all trivial. This
  shows that the dimension of the space
  $\Ext^1_{\mathcal{H}}(M(i,r), M(i,r))$ is $2$ from our calculations.
 \end{proof}
\begin{lemma}\label{iwahori_hecke_ext}
  The space of extensions $\Ext^1_{\mathcal{H}}(M(i,0), M(j,0))$ is
  trivial for $i=j$ and has dimension $1$ for $i\neq j$.
\end{lemma}
\begin{proof}
  The algebra $e_1\mathcal{H}_{v_i}$ is semi-simple algebra and hence
  we get that
  \begin{equation}\label{iwahori_trivial_ext}
\Ext^i_{\mathcal{H}_{v_i}}(\chi_{\id, S}, \chi_{\id, S'})=0
\end{equation}
for all $i>0$ and for subsets $S$ and $S'$ of $\{s_0, s_1\}$.  Now
consider the exact sequence \eqref{ext_two_points} when $\mathfrak{m}$
is $M(i,r)$ and $\mathfrak{n}$ is $M(j,r)$.  For $i=j$ the map
$\delta$ in \eqref{ext_two_points} is zero hence the space
$\Ext^1_{\mathcal{H}}(M(i,r), M(i,r))$ is trivial. When
$i\neq j$ the first two $\ho$ spaces in \eqref{ext_two_points} are
trivial. The space $\ho_{\mathcal{H}_E}(\mathfrak{m}, \mathfrak{n})$
has dimension one. This shows that the dimension of the space
$\Ext^1_{\mathcal{H}}(M(i,r), M(i,r))$ is $1$ for $i\neq j$.
\end{proof}
\section{The Hecke module
  \texorpdfstring{$\mathbb{R}^1\mathcal{I}(\pi_{i,r})$}{}.}
Pa\v{s}k\={u}nas calculated the cohomology groups
$\mathbb{R}^1\mathcal{I}(\pi_{i,r})$ and we now recall his
results. Let $\tilde{\pi}_r$ be the supersingular representation
$\pi(\sigma_r, \mu_1)$ of $G$. Recall that the $K$-socle of
$\tilde{\pi}_r$ is isomorphic to $\sigma_r\oplus \sigma_{p-1-r}$ and
the space of $I(1)$ invariants has a basis $(\bv_0, \bv_1)$
where $\bv_{0}$ and $\bv_{1}$ belong to $\sigma_r^{N_p}$ and
$\sigma_{p-1-r}^{N_p}$ respectively. Let $I^+$ and $I^-$ be the groups
$I\cap U$ and $I\cap \bar{U}$ respectively. Consider the spaces
\begin{align*}
  M_0:=<I^{+}\theta^n\bv_1; \ n\geq 0>\ \text{and}\
  M_1:=<I^{+}\theta^n\bv_2; \ n\geq 0>
\end{align*}  
and let $\Pi$ be the matrix $\begin{pmatrix}0&1\\p&0\end{pmatrix}$
which normalizes $I$ and $I(1)$. We denote by $\pi_0$ and $\pi_1$ the
spaces $M_0+\Pi M_1$ and $M_1+\Pi M_0$.  Let $G^{0}$ be subgroup of
$G$ consisting of elements with integral discriminant. Let
$G^{+}$ be the group $ZG^0$.  We denote by $Z_1$ the group
$I(1)\cap Z$.
\begin{proposition}[Pa\v{s}k\={u}nas]\label{paskunas_presentation}
  The spaces $\sigma_0$ and $\sigma_1$ are $G^{+}$ stable. The space
  $\tilde{\pi}_r$ is the direct sum of the representations $\pi_1$ and
  $\pi_0$ as $G^{+}$ representations and hence $\pi_{i,r}$ is
  isomorphic to $\pi_i$ as $G_S$ representations for $i\in\{0,1\}$. If
  $r$ be an integer such that $0<r< (p-1)/2$ then the Hecke module
  $\mathbb{R}^1\mathcal{I}(\pi_{i,r})$ is isomorphic to
  $\mathcal{I}(\pi_{i,r})\oplus\mathcal{I}(\pi_{i,r})$. In the Iwahori
  case (i.e, $r=0$) the Hecke module
  $\mathbb{R}\mathcal{I}(\pi_{0,0})\oplus\mathbb{R}\mathcal{I}(\pi_{1,0})$
  is isomorphic to
  $\mathcal{I}(\pi_{0,0})^{\oplus 2}\oplus
  \mathcal{I}(\pi_{0,0})^{\oplus 2}$.
\end{proposition}
\begin{proof}
  The first part follows from \cite[Corollary
  6.5]{paskunas_extensions}. The second part follows from
  \cite[Proposition 10.5, Theorem 10.7 and equation
  (49)]{paskunas_extensions}.
\end{proof}
\begin{corollary}
  Let $\tau$ be an irreducible supersingular representation of
  $G_S$. If the space of extensions $\Ext^1_{G_S}(\tau, \pi_{i,r})$ is
  non-trivial then $\tau\simeq \pi_{j,r}$ for some $j\in \{0,1\}$.
\end{corollary}
\begin{proof}
  This follows from \eqref{fiveterm}, \eqref{distinct_blocks} and
  Proposition \ref{paskunas_presentation}.
\end{proof}
\begin{corollary}\label{cross_extensions}
  Let $0<r<(p-1)/2$ and $i\neq j$ then the dimensions of the space
  $\Ext_{G_S}^1(\pi_{i,r}, \pi_{j,r})$ is $2$. 
\end{corollary}
\begin{proof}
  Observe that for $0<r<(p-1)/2$ the modules $M(i,r)$ and $M(j,s)$ are
  not isomorphic. Now using the exact sequence \eqref{fiveterm} and
  Proposition \ref{paskunas_presentation} we get that
$$\Ext^1_{G_S}(\pi_{i,r}, \pi_{j,r})
\simeq\Ext^1_{\mathcal{H}}(M(i,r), M(j,r)).$$ The corollary follows
from the Lemma \ref{regular_hecke_ext}.
\end{proof}
\begin{remark}
  The results of Corollary \ref{cross_extensions} remain valid for
  $r=0$ but we prove this later. It is interesting to note that for
  $0<r<(p-1)/2$ and $i\neq j$ any extension $E$ of $\pi_{i,r}$ by
  $\pi_{j,r}$ for $i\neq j$ is generated by its $I(1)_S$ invariants,
  i.e, $E=<G_SE^{I(1)_S}>$. 
\end{remark}
\section{Calculation of degree one self extensions.}
Let us first consider the case when $0<r\leq(p-1)/2$. In order to
determine the dimensions of $\Ext^1(\pi_{i,r}, \pi_{i,r})$ we first
show that the map
\begin{equation}\label{delta_map}
  \Ext^1_{G_S}(\pi_{i, r}, \pi_{i,r})\rightarrow
  \ho_{\mathcal{H}}(\mathcal{I}(\pi_{i,r}),
  \mathbb{R}^1\mathcal{I}(\pi_{i,r})))
\end{equation} is non-zero. Explicitly the above map takes 
an extension $E$, with $0\rightarrow \pi_{i, r}\rightarrow E\rightarrow \pi_{i,
  r}\rightarrow 0$, to the delta map in the associated long exact
sequence, given by:
$\mathcal{I}(\pi_{i,
  r})\xrightarrow{\delta_E}\mathbb{R}^1\mathcal{I}(\pi_{i, r})$. 
Note that the dimension of $E^{I(1)}$ is one if and only if
$\delta_E\neq 0$. 
\begin{lemma}
  For $0<r\leq(p-1)/2$ then map \eqref{delta_map} is non-zero.
\end{lemma}
\begin{proof}
  For $0<r\leq(p-1)/2$ there exists a self extension $E$ of
  $\tilde{\pi}_r$ such that the map
  $\mathcal{I}(\tilde{\pi}_{r})\xrightarrow{\delta_E}
  \mathbb{R}^1\mathcal{I}(\tilde{\pi}_{r})$ is non-zero. We fix an
  extension $E$ such that $\delta_E\neq 0$. Since $\delta_E$ is a
  Hecke-equivariant map and $\mathcal{I}(\tilde{\pi}_r)$ is an
  irreducible Hecke-module of dimension $2$ we get that the inclusion
  map of $\mathcal{I}(\tilde{\pi}_r)$ in $\mathcal{I}(E)$ is an
  isomorphism i.e, $\dim E^{I(1)}=2$. Now consider the pullback diagram
\begin{equation}\label{pullback}
\begin{tikzpicture}[node distance=1.5cm,auto]
\node(A_1){$0$};
\node(B_1)[right of=A_1]{$\tilde{\pi}_r$};
\node(C_1) [right of=B_1] {$E_1$};
\node(D_1) [right of=C_1] {$\pi_{i,r}$};
\node(E_1) [right of=D_1] {$0$};
\node(A_2)[below of =A_1]{$0$};
\node(B_2)[below of=B_1]{$\tilde{\pi}_r$};
\node(C_2) [below of=C_1] {$E$};
\node(D_2) [below of=D_1] {$\tilde{\pi}_r$};
\node(E_2) [right of=D_2] {$0$.};

\draw[->] (A_1) -- (B_1);
\draw[->] (B_1) -- (C_1);
\draw[->](C_1) -- (D_1);
\draw[->](D_1) -- (E_1);
\draw[->](A_2) -- (B_2);
\draw[->](B_2) -- (C_2);
\draw[->] (C_2) -- (D_2);
\draw[->](D_2) -- (E_2);
\draw[->](B_1) -- (B_2);
\draw[->](C_1) -- (C_2);
\draw[->](D_1) -- (D_2);
\end{tikzpicture} 
\end{equation}
The long exact sequences in $I(1)$-group cohomology attached to
\eqref{pullback} gives us:
\begin{center}
$\begin{tikzpicture}[node distance=2cm, auto]
\node(A_1){$0$};
\node(B_1)[right of=A_1]{$\mathcal{I}(\tilde{\pi}_r)$};
\node(C_1) [right of=B_1] {$\mathcal{I}(E_1)$};
\node(D_1) [right of=C_1] {$\mathcal{I}(\pi_{i,r})$};
\node(E_1) [right of=D_1] {$\mathbb{R}^1\mathcal{I}(\tilde{\pi}_r)$};
\node(A_2)[below of =A_1]{$0$};
\node(B_2)[below of=B_1]{$\mathcal{I}(\tilde{\pi}_r)$};
\node(C_2) [below of=C_1] {$\mathcal{I}(E)$};
\node(D_2) [below of=D_1] {$\mathcal{I}(\tilde{\pi}_r)$};
\node(E_2) [right of=D_2] {$\mathbb{R}^1\mathcal{I}(\tilde{\pi}_r)$.};
\draw[->] (A_1) -- (B_1);
\draw[->] (B_1) to node[above]{$f$}(C_1);
\draw[->](C_1) -- (D_1);
\draw[->](D_1) to node[above]{$\delta_2$}(E_1);
\draw[->](A_2) -- (B_2);
\draw[->](B_2) -- (C_2);
\draw[->] (C_2) -- (D_2);
\draw[->](D_2) to node[below]{$\delta_1$}(E_2);
\draw[->](B_1) -- (B_2);
\draw[->](C_1) -- (C_2);
\draw[->](D_1) -- (D_2);
\draw[->](E_1) -- (E_2);
\end{tikzpicture} $
\end{center}
Since the dimension of $\mathcal{I}(E)$ is $2$ we
get that $\delta_1$ is injective and hence the map $\delta_2$
is non-zero. The dimension of the space $\mathcal{I}(\pi_{i,r})$ is
one hence $f$ is an isomorphism. This shows
that the space $\mathcal{I}(E_1)$ has dimension $2$. For $r=(p-1)/2$
the representations $\pi_{1,r}\simeq \pi_{0,r}$. We assume without
loss of generality $\img\delta_2$ is contained in
$\mathbb{R}^1\mathcal{I}(\pi_{i,r})$.  For any $r$ such that
$0<r\leq (p-1)/2$ consider the pushout of $\tilde{\pi}_r$ by
$\pi_{i,r}$
\begin{equation}\label{pushforward}
\begin{tikzpicture}[node distance=1.5cm,auto]
\node(A_1){$0$};
\node(B_1)[right of=A_1]{$\pi_{i,r}$};
\node(C_1) [right of=B_1] {$E_2$};
\node(D_1) [right of=C_1] {$\pi_{i,r}$};
\node(E_1) [right of=D_1] {$0$};
\node(A_2)[below of =A_1]{$0$};
\node(B_2)[below of=B_1]{$\tilde{\pi}_r$};
\node(C_2) [below of=C_1] {$E_1$};
\node(D_2) [below of=D_1] {$\pi_{i,r}$};
\node(E_2) [right of=D_2] {$0$};

\draw[->] (A_1) -- (B_1);
\draw[->] (B_1) -- (C_1);
\draw[->](C_1) -- (D_1);
\draw[->](D_1) -- (E_1);
\draw[->](A_2) -- (B_2);
\draw[->](B_2) -- (C_2);
\draw[->] (C_2) -- (D_2);
\draw[->](D_2) -- (E_2);
\draw[->](B_1) -- (B_2);
\draw[->](C_1) -- (C_2);
\draw[->](D_1) -- (D_2);
\end{tikzpicture} 
\end{equation}
The self extension $E_2$ of $\pi_{i,r}$ is non-split
and the induced map $\delta_{E_2}$ is non-zero. To see this consider the
long exact sequence in cohomology attached to  \eqref{pushforward}:
\begin{center}
$\begin{tikzpicture}[node distance=2cm,auto]
\node(A_1){$0$};
\node(B_1)[right of=A_1]{$\mathcal{I}(\pi_{i,r})$};
\node(C_1) [right of=B_1] {$\mathcal{I}(E_2)$};
\node(D_1) [right of=C_1] {$\mathcal{I}(\pi_{i,r})$};
\node(E_1) [right of=D_1] {$\mathbb{R}^1\mathcal{I}(\pi_{i,r})$};
\node(A_2)[below of =A_1]{$0$};
\node(B_2)[below of=B_1]{$\mathcal{I}(\tilde{\pi}_r)$};
\node(C_2) [below of=C_1] {$\mathcal{I}(E_1)$};
\node(D_2) [below of=D_1] {$\mathcal{I}(\pi_{i,r})$};
\node(E_2) [right of=D_2] {$\mathbb{R}^1\mathcal{I}(\tilde{\pi}_r)$};
\draw[->] (A_1) -- (B_1);
\draw[->] (B_1) to node[above]{$g$}(C_1);
\draw[->](C_1) -- (D_1);
\draw[->](D_1) to node[above]{$\delta_3$}(E_1);
\draw[->](A_2) -- (B_2);
\draw[->](B_2) to node[above]{$\simeq$}(C_2);
\draw[->] (C_2) to node[above]{$0$}(D_2);
\draw[->](D_2) to node[above]{$\delta_2$}(E_2);
\draw[->](B_2) -- (B_1);
\draw[->](C_2) -- (C_1);
\draw[->](D_2) -- (D_1);
\draw[->](E_2) to node[right]{$\mathbb{R}^1\mathcal{I}(p_2)$}(E_1);
\end{tikzpicture} $
\end{center}
Note that $\mathbb{R}^1\mathcal{I}(\tilde{\pi}_r)$ is isomorphic to
$\mathbb{R}^1\mathcal{I}(\pi_{0,r})\oplus
\mathbb{R}^1\mathcal{I}(\pi_{1,r})$ and $\mathbb{R}^1\mathcal{I}(p_2)$
is the projection map. This shows that
$\mathbb{R}^1\mathcal{I}(p_2)\delta_2\neq 0$ and hence
$\delta_3\neq 0$ using which we get that $g$ is an isomorphism. This
shows that $E_2$ is a non-split self-extension of $\pi_{i, r}$ by
$\pi_{i,r}$.
\end{proof}
\begin{corollary}\label{first_bound}
  For any integer $r$ such that $0<r< (p-1)/2$ we have
  $\dim_{\bar{\mathbb{F}}_p}\Ext^1_G(\pi_{i, r}, \pi_{i,r})\geq 1$. 
\end{corollary}

\begin{theorem}\label{main_theorem_final}
  Let $p\geq 5$ and $0\leq r< (p-1)/2$ then the dimension of
  $\Ext_{G_S}^1 (\pi_{i,r}, \pi_{i,r})$ is $1$ and dimension of
  $\Ext_{G_S}^1 (\pi_{i,r}, \pi_{j,r})$ is $2$ for $i\neq j$.  For
  $r=(p-1)/2$ the dimension of $\Ext_{G_S}^1 (\pi_{0,r}, \pi_{0,r})$
  is $3$.
\end{theorem}
\begin{proof}
  The subgroup $G_SZ$ is an index $2$ subgroup of $G$ and $\id$ and
  $\Pi$ are two double coset representatives for $K\backslash G/G_SZ$.
  We note that $K\cap G_S$ and $K^{\Pi}\cap G_S$ are representatives
  for the two distinct classes of maximal compact subgroups of $G_S$
  and we denote them by $K_1$ and $K_2$ respectively. Let $\sigma_r'$
  be the representation $\sigma_r^{\Pi}$ of $K^{\Pi}$. Using
  Mackey-decomposition we get that
\begin{equation}\label{gl_2-sl_2-res}
  \res_{G_S}\ind_{KZ}^{G}\sigma_r=
  \ind_{K^{\Pi}\cap G_S}^{G_S}\sigma_r^{\Pi}\oplus \ind_{K\cap G_S}^{G_S}\sigma_r
=\ind_{K_1}^{G_S}\sigma_r\oplus \ind_{K_2}^{G_S}\sigma_r'.
\end{equation}
using  the long exact sequence of $\Ext$ groups for the exact
sequence, 
$$0\rightarrow \ind_{ZK}^{G}\sigma_r\xrightarrow{T}
\ind_{ZK}^{G}\rightarrow \tilde{\pi}_r\rightarrow 0$$
we get that an exact sequence 
\begin{equation}\label{sandwich_sequence}
  \ho_{G}(\ind_{ZK}^{G}\sigma_r, \tilde{\pi}_r)\rightarrow
  \Ext^1_G(\tilde{\pi}_r, \tilde{\pi}_r)\rightarrow
  \Ext^1_G(\ind_{ZK}^{G}\sigma_r, \tilde{\pi}_r)
  \xrightarrow{T}\Ext^1_G(\ind_{ZK}^G\sigma_r, 
  \tilde{\pi}_r).
\end{equation} 
Now using \eqref{gl_2-sl_2-res} the exact sequence
\eqref{sandwich_sequence} becomes
\begin{equation}\label{sandwich_sequence_1}
  0\rightarrow \ho_{K_1}(\sigma_r, \tilde{\pi}_r)\oplus
  \ho_{K_2}(\sigma_r', \tilde{\pi}_r)\rightarrow
  \Ext^1_G(\tilde{\pi}_r, \tilde{\pi}_r)\rightarrow
  \Ext^1_{K_1}(\sigma_r, \tilde{\pi}_r)\oplus
  \Ext^1_{K_2}(\sigma_r', \tilde{\pi}_r).
\end{equation}
The groups $K_1$ is contained in $K/Z_1$. For all $i\geq 0$ we note
that
$$\Ext^i_{K_1}(\sigma_r, \tilde{\pi}_r)\simeq 
\Ext^i_{K/Z_1}(\ind_{K_1}^{K/Z_1}(\sigma_r), \tilde{\pi}_r) \simeq
\bigoplus_{0\leq a<p-1}\Ext^i_{K/Z_1}(\sigma_r\otimes{\rm det}^{a},
\tilde{\pi}_r).$$ 

The spaces
$\Ext^1_{K/Z_1}(\sigma_r\otimes{\rm det}^{a}, \tilde{\pi}_r)$ can be
calculated from the work of Pa\v{s}k\={u}nas. We recall his
calculations as needed. There exists a $G$ smooth representation
$\Omega$ such that $\res_K\Omega$ is an injective envelope of
${\rm Soc}_K(\tilde{\pi}_r)$ in the category of smooth representations
of $K$. In particular we get that $\tilde{\pi}_r$ is contained in
$\Omega$.  The restriction $\res_K\Omega$ is isomorphic to
${\rm inj}\sigma_r\oplus {\rm inj}\sigma_{p-1-r}$. Now
$\Ext^1_{K/Z_1}(\sigma_r\otimes{\rm det}^a, \tilde{\pi}_r)$ is
isomorphic to
$\ho_{K/Z_1}(\sigma_r\otimes{\rm det}^a, \Omega/\tilde{\pi}_r)$.  

{\bf We now use the notations from \cite[Notations, Section
  9]{paskunas_extensions}}. We make one modification. Pa\v{s}k\={u}nas
uses the notation $\chi$ for the character
$$\begin{pmatrix}[\lambda]&0\\0&[\mu]\end{pmatrix}
\mapsto (\lambda)^r(\lambda\mu)^a$$ for all
$\lambda, \mu \in \mathbb{F}_p^{\times}$ and $[\ ]$ is the Teichmuller
lift.  For convenience we use the notation $\chi_{a,r}$ instead of
$\chi$. The idempotent $e_{\chi}$ in \cite[Section
9]{paskunas_extensions} will be denoted $e_{\chi_{a,r}}$.  The space
$\ho_{K_1}(\sigma_r\otimes{\rm det}^a, \Omega/\tilde{\pi}_r)$ is the
same as
\begin{equation}\label{ext_cpt_hom}
\ker(\mathcal{I}(\Omega/\tilde{\pi}_r)e_{\chi_{r,a}}
\xrightarrow{T_{n_s}}\mathcal{I}(\Omega/\tilde{\pi}_r)e_{\chi_{r,a}^s})
\end{equation}
and from \cite[Proposition 10.10]{paskunas_extensions} has dimension
less than or equal to $2$. For $0\leq r\leq (p-1)/2$ the space
$\ho_{K/Z_1}(\sigma_r\otimes{\rm det}^a, \tilde{\pi}_r)$ is non-zero
if and only if $a=0$ and has dimension $1$ if $r<(p-1)/2$ and $2$
otherwise. Using \eqref{ext_cpt_hom} for $0\leq r<(p-1)/2$ the space
$\Ext^1_{K/Z_1}(\sigma_r\otimes{\rm det}^a, \tilde{\pi}_r)$ is
non-zero if and only if $a=0$ and has dimension at most $2$ (see
\cite[Proposition 10.10]{paskunas_extensions} for $0<r<(p-1)/2$ and
\cite[Corollary 6.13 and Corollary 6.16]{paskunas_colmez_functor} for
$r=0$). When $r=(p-1)/2$ the space
$\Ext^1_{K/Z_1}(\sigma_r\otimes{\rm det}^a, \tilde{\pi}_r)$ is
non-zero for $a=0$ and $a=(p-1)/2$ and in each of these cases the
dimension of the space
$\Ext^1_{K/Z_1}(\sigma_r\otimes{\rm det}^a, \tilde{\pi}_r)$ is less
than or equal to $2$.

Now using exact sequence \eqref{sandwich_sequence_1} the space
$\Ext^1_{G_S}(\tilde{\pi}_r, \tilde{\pi}_r)$ has dimension less than
or equal to $6$ for $0\leq r<(p-1)/2$ and its dimension is less than
or equal to $12$ if $r=(p-1)/2$. For $r\neq 0$ using this upper bound
and the lower bounds from Corollary \ref{first_bound} and Corollary
\ref{cross_extensions} we deduce the theorem in this case. When $r=0$
Pa\v{s}k\={u}nas showed that (see \cite[Proposition
6.15]{paskunas_extensions}) the dimension of
$\Ext_{G^{+}/Z}^1(\pi_{i,0}, \pi_{j,0})$ is $2$ when $i\neq j$ and $1$
otherwise. Since $G_S/\{\pm 1\}$ has index a factor of $2$ in $G^+/Z$
and $G_S\cap Z$ acts trivially on $\pi_{i,0}$ we get that
\begin{equation}\label{lower_bounds_iwahori}
    \Ext_{G^{+}/Z}^1(\pi_{i,0}, \pi_{j,0})\hookrightarrow
    \Ext_{G_S/\{\pm1\}}^1(\pi_{i,0}, \pi_{j,0})=\Ext_{G_S}^1(\pi_{i,0},
    \pi_{j,0}).
\end{equation}
From our upper bounds the inclusions \eqref{lower_bounds_iwahori} are
strict and hence we prove the theorem.
\end{proof}
\begin{corollary}
The Hecke module $\mathbb{R}^1\mathcal{I}(\pi_{i,0})$ is isomorphic to
the module $\mathcal{I}(\pi_{i,0})\oplus\mathcal{I}(\pi_{j,0})$ for
$i\neq j$. 
\end{corollary}
\begin{proof}
From the Theorem \ref{main_theorem_final}, exact sequence
\eqref{fiveterm} and \eqref{iwahori_trivial_ext} we get that dimension
of the space $\ho_{\mathcal{H}}(\mathcal{I}(\pi_{0,0}),
\mathbb{R}^1\mathcal{I}(\pi_{0,0}))$ is $1$. Using the Proposition
\ref{paskunas_presentation} we get that   
$$\mathbb{R}^1\mathcal{I}(\pi_{i,0})\simeq
\mathcal{I}(\pi_{i,0})\oplus \mathcal{I}(\pi_{j,0}).$$
\end{proof}
\bibliography{../biblio} \bibliographystyle{amsalpha}
\noindent Santosh Nadimpalli,\\
School of Mathematics, Tata Institute of Fundamental Research,
Mumbai, 400005.\\
\texttt{nvrnsantosh@gmail.com}, \texttt{nsantosh@math.tifr.res.in}
\end{document}